\documentclass{article}
\usepackage{amsmath,amssymb}
\usepackage{amsthm}

\newtheorem{dfn}{Definition}[section]
\newtheorem{lem}{Lemma}[section]
\newtheorem{thm}{Theorem}[section]

\newtheorem{rmk}{Remark}[section]

\title{A Lower Bound on WAFOM}
\author{Takehito Yoshiki
\footnote{Tokyo 153-8914  Graduate School of Mathematical Sciences}
\footnote{Email: yosiki@ms.u-tokyo.ac.jp}
\footnote{This work was supported by the Program for Leading Graduate Schools, MEXT, Japan.}
}

\begin{document}
\maketitle
\begin{abstract}
We give a lower bound on Walsh figure of merit (WAFOM), which
is a parameter to estimate 
the integration error for quasi-Monte Carlo (QMC) integration
by a point set called a digital net. 
This lower bound is optimal because
the existence of point sets attaining the order was proved
in [K. Suzuki, An explicit construction of point sets with
large minimum Dick weight, Journal of Complexity 30, (2014), 347-354].
\end{abstract}

\section{Introduction}

We explain the relation between quasi-Monte Carlo (QMC) integration
 and the Walsh figure of merit (WAFOM) (see \cite{1} for details).
QMC integration is one of the methods for numerical integration
 (see \cite{4}, \cite{5} and \cite{8} for details). 
Let $Q$ be a point set in the $s$-dimensional cube $[0,1)^s$ 
with finite cardinality $\# (Q)=N$, and $f:[0,1)^s \rightarrow \mathbb{R}$ be a Riemann integrable function. 
The QMC integration by $Q$ is the approximation of $I(f):=\int_{[0,1)^s}f(x)dx$
by the average $I_{Q}(f):=\frac{1}{\#(Q)}\sum_{x \in Q}f(x)$. 

WAFOM bounds the error of QMC integration for a certain class of functions by a point set $P$
called a digital net, which is
defined by the following identification (see \cite{1} and \cite{5} for details): 
Let ${\mathcal P}$ be a subspace of 
$s \times n$ matrices over the finite field $\mathbb{F}_2$ of order two.
We define the function 
$\varphi :{\mathcal P} \ni X=(x_{i,j}) \mapsto 
x=(\sum_{j=1}^{n}x_{i,j} \cdot 2^{-j})_{i=1}^{s}\in\mathbb{R}^s,$
 where $x_{i,j}$ is considered to be 0 or 1 in $\mathbb{Z}$ and the sum is taken in ${\mathbb R}$. 
The digital net $P$ in $[0,1)^s$ is defined by $\varphi ( {\mathcal P})$.
We identify the digital net $P$ with a linear space ${\mathcal P}$.
If ${\mathcal P}$ is an $m$-dimensional space, the cardinality of $P$ is $2^m$.

Let $f$ be a function whose mixed partial
derivatives up to order $\alpha\geq 1$ in each variable are square integrable (see \cite[3]{3} for details). 
We say that such a function $f$ is an $\alpha$-smooth function or the smoothness of a function $f$ is $\alpha$ here.
By using `$n$-digit discretization $f_n$' (see \cite{1} for details), we approximate $I(f)$ by 
$I_{P}(f_n):=\frac{1}{\#(P)}\sum_{x \in P} f_n(x)$ for an $n$-smooth function $f$, 
that is, we can evaluate the integration error 
by the following Koksma-Hlawka type inequality of WAFOM: 
\[
\left| I(f)-I_{P}(f_n) \right|
\leq 
C_{s,n}||f||_n \times {\rm WAFOM}(P),
\]
where $||f||_n$
is the norm of $f$ defined in \cite{3}
and $C_{s,n}$ is a constant independent of $f$ and $P$. 
If the difference between $I_{P}(f_n)$ and $I_{P}(f)$ is 
negligibly small, we see that
$\left| I(f)-I_{P}(f) \right|
\leq 
C_{s,n}||f||_n \times {\rm WAFOM}(P)
$ approximately holds (see \cite{1} for details).
In \cite{2}, we proved that there is a digital net 
$P$ of size $2^m$ with ${\rm WAFOM}(P)<2^{-Cm^2/s}$ for sufficiently large $m$
 by a probabilistic argument. (Suzuki \cite{6} gave a constructive proof.)
In this paper, we prove that
$\mathrm{WAFOM}(P) > 2^{-C'm^2/s}$ holds for large $m$ and any digital net $P$ with $\# (P)=2^m$
(see Theorem \ref{low} for a precise statement, which 
is formulated for a linear subspace ${\mathcal P}$, instead of a digital net $P$).
Thus this order is optimal.

This paper is organized as follows: 
We introduce some definitions in Section~2. We prove a lower bound on WAFOM in Section 3. 

\section{Definition and notation}
\label{sec2}
In this section, we introduce WAFOM 
and the minimum weight which will be needed later on.

Let $s$ and $n$ be positive integers. 
${\rm M}_{s,n}(\mathbb{F}_2)$ denotes the set of $s \times n$ matrices over 
the finite field $\mathbb{F}_2$ of order 2. 
We regard ${\rm M}_{s,n}(\mathbb{F}_2)$ as an $sn$-dimensional inner product space under
the inner product 
$A\cdot B=(a_{i,j})\cdot (b_{i,j})=\sum_{i,j}a_{i,j}b_{i,j} \in \mathbb{F}_2$. 

WAFOM is defined using a Dick weight in \cite{1}.

\begin{dfn}
\normalfont
Let $X=(x_{i,j})$ be an element of ${\rm M}_{s,n}(\mathbb{F}_2)$.
The Dick weight of $X$ is defined by 
\begin{align*}
\mu (X):= \sum_{1\leq i\leq s,1\leq j\leq n} j\cdot x_{i,j},
\end{align*}
where we regard $x_{i,j} \in \{ 0, 1 \}$ as the element of $\mathbb{Z}$
and take the sum in $\mathbb{Z}$, not in $\mathbb{F}_2$.
\end{dfn}

\begin{dfn}
\normalfont
Let ${\mathcal P}$ be a subspace of ${\rm M}_{s,n}(\mathbb{F}_2)$. 
WAFOM of ${\mathcal P}$ is defined by 
\begin{align}
{\rm WAFOM}({\mathcal P}):= \sum_{X \in {\mathcal P}^\bot \backslash \{O\}}2^{-\mu(X)},
\label{WF}
\end{align}
where ${\mathcal P}^\bot$ denotes the orthogonal space to ${\mathcal P}$ in ${\rm M}_{s,n}(\mathbb{F}_2)$
and $O$ denotes the zero matrix.
\end{dfn}

In order to estimate a lower bound on WAFOM, we use the minimum weight introduced in \cite{2}.

\begin{dfn}
\normalfont
Let ${\mathcal P}$ be a proper subspace of ${\rm M}_{s,n}(\mathbb{F}_2)$.
 The minimum weight of ${\mathcal P}^\bot$ is defined by
\begin{align}
\label{defdelta}
\delta_{{\mathcal P}^\bot}:=\min_{ X \in {\mathcal P}^\bot \backslash \{ O \}} \mu (X).
\end{align}
\end{dfn}

\section{A lower bound on WAFOM}
Now we state a lower bound on WAFOM.
The theorem is mentioned for a linear subspace identified
with a digital net (see Section 1).

\begin{thm}
\label{low}
Let $n$, $s$ and $m$ be positive integers such that $m<ns$, 
and let $C'$ be an arbitrary real number greater than $1/2$. 
If $m/s\geq (\sqrt{C'+1/16}+3/4)/(C'-1/2)$,
then for any $m$-dimensional subspace ${\mathcal P}$ of ${\rm M}_{s,n}(\mathbb{F}_2)$
we have
\begin{align*}
{\rm WAFOM}({\mathcal P}) \geq 2^{-C'm^2/s}.
\end{align*}
\end{thm}
\begin{proof}
Let $n,s,m$ and $C'$ be defined as above.
The following inequality immediately results from (\ref{WF}), (\ref{defdelta}) in Section \ref{sec2}:
\begin{align}
{\rm WAFOM}({\mathcal P})=\sum_{X \in {\mathcal P}^\bot \backslash \{ O \}}
2^{-\mu (X)} \geq 2^{-\delta_{{\mathcal P}^\bot}}.
\label{WFdelta}
\end{align}
By an upper bound on $\delta_{{\mathcal P}^\bot}$ in Lemma \ref{lem} (b) below
 and the inequality (\ref{WFdelta}), 
for any $m$-dimensional subspace ${\mathcal P}$ of ${\rm M}_{s,n}(\mathbb{F}_2)$, we have
\begin{align*}
{\rm WAFOM}({\mathcal P})&=
\sum_{X \in {\mathcal P}^\bot \backslash \{ O \}} 2^{-\mu (X)} \geq 2^{-\delta_{{\mathcal P}^\bot}}
\geq 2^{-C' m^2/s}.
\end{align*}
Thus Theorem \ref{low} follows.
\end{proof}

We prove an upper bound on the minimum weight $\delta_{{\mathcal P}^\bot}$
 to complete the proof of Theorem \ref{low}.

\begin{lem}
\label{lem}
Let $n$, $s$ and $m$ be positive integers such that $m<ns$.
Then we have the following statements:

\begin{itemize}
\item[{\rm \bf (a)}]
Let q and r be non-negative integers satisfying $q=(m-r)/s$ and $r<s$. 
Then we obtain
\begin{align*}
\delta_{{\mathcal P}^\bot} \leq \frac{sq(q+1)}{2}+(q+1)(r+1) 
\end{align*}
for any $m$-dimensional subspace ${\mathcal P}$ of ${\rm M}_{s,n}(\mathbb{F}_2)$.

\item[{\rm \bf (b)}]
Let $C'$ be an arbitrary positive real number greater than $1/2$. If 
$m/s\geq (\sqrt{C'+1/16}+3/4)/(C'-1/2)$, then we have 
\begin{align*}
\delta_{{\mathcal P}^\bot} \leq C' m^2/s
\end{align*}
for any $m$-dimensional subspace ${\mathcal P}$ of ${\rm M}_{s,n}(\mathbb{F}_2)$.
\end{itemize}
\end{lem}
\begin{proof}
{\bf (a)} If there exists a subspace $W$ of ${\rm M}_{s,n}(\mathbb{F}_2)$ such that 
for any $m$-dimensional subspace ${\mathcal P}$ of ${\rm M}_{s,n}(\mathbb{F}_2)$
we have ${\mathcal P}^\bot \cap W \neq \{ O \}$, 
then $\delta_{{\mathcal P}^\bot} \leq \max_{X \in W} \mu(X)$ holds.
Therefore in order to obtain a sharp upper bound on $\delta_{{\mathcal P}^\bot}$, 
we need a subspace $W$ with $\max_{X \in W} \mu(X)$ small.
We can construct $W$ as follows:
\begin{eqnarray*}
W:=
\left \{
X=(x_{i,j}) \in {\rm M}_{s,n}(\mathbb{F}_2) \ \ \middle| \ \ 
x_{i,j}=0
\begin{array}{l}
 \ (i \leq r+1 \ \text{and} \ q+2\leq j ) \\
 \hspace{2cm}\mbox{or} \\
 \ (r+2\leq i \ \text{and} \ q+1\leq j)
\end{array}
\right \} ,
\end{eqnarray*}
that is, $W$ consists of the following type of matrices:
\begin{eqnarray}
X = \left(
    \begin{array}{ccccccc}
      x_{1,1}   & \ldots & x_{1,q}  & x_{1,q+1}  & 0 & \ldots & 0 \\
      \vdots    & \vdots & \vdots   & \vdots     & 0 & \ldots & 0 \\
      x_{r+1,1} & \ldots & x_{r+1,q}& x_{r+1,q+1}& 0 & \ldots & 0 \\
      x_{r+2,1} & \ldots & x_{r+2,q}&   0        & 0 & \ldots & 0 \\
      \vdots    & \vdots & \vdots   & \vdots     & 0 & \ldots & 0 \\
      x_{s,1}   & \ldots & x_{s,q}  &   0        & 0 & \ldots & 0 \\
          \end{array}
  \right) \ \ \ \ \ 
( \ x_{i,j} \in \mathbb{F}_2 \ ).
\label{W}
\end{eqnarray}
The subspace $W$ satisfies
${\mathcal P}^\bot \cap W \neq \{ O \} $
for any $m$-dimensional subspace ${\mathcal P}$ of ${\rm M}_{s,n}(\mathbb{F}_2)$.
Indeed we can see that
\begin{align*}
\text{dim}({\mathcal P}^\bot \cap W) &\geq
\text{dim}{\mathcal P}^\bot+\mbox{dim}W-\mbox{dim}{\rm M}_{s,n}(\mathbb{F}_2) \\
&=(sn-m)+(sq+r+1)-sn=1.
\end{align*}
Hence there exists a non-zero matrix $X_{{\mathcal P}} \in W \cap
{\mathcal P}^\bot$. This yields
\[
\delta_{{\mathcal P}^\bot} = \min_{ X \in {\mathcal P}^\bot \backslash \{ O \}} \mu (X) 
\leq  \mu(X_{{\mathcal P}}) \leq \max_{X \in W} \mu(X).
\]
Let us estimate $\max_{X \in W} \mu(X)$ of $W$. 
Let $X_{\text{max}}$ of $W$ be a matrix whose 
entries $x_{i,j}$ in (\ref{W}) are all 1.
The function $\mu$ attains its maximum at 
$X_{\text{max}}$ in $W$. 
Thus it follows that
\[
\max_{X \in W} \mu(X)=\mu(X_{\text{max}})=\dfrac{sq(q+1)}{2}+(q+1)(r+1).
\]

We obtain that
\begin{align*}
\delta_{{\mathcal P}^\bot} = \min_{ X \in {\mathcal P}^\bot \backslash \{ O \}} \mu (X) 
\leq  \mu(X_{{\mathcal P}}) \leq \max_{X \in W} \mu(X)
 = \frac{sq(q+1)}{2}+(q+1)(r+1),
\end{align*}
where ${\mathcal P}$ is an arbitrary $m$-dimensional subspace of ${\rm M}_{s,n}(\mathbb{F}_2)$. 

{\bf (b)}
Let $C'$ be a real number greater than $1/2$ and assume $m/s\geq (\sqrt{C'+1/16}+3/4)/(C'-1/2)$.
By combining $r+1\leq s$, $q \leq m/s$ and the assertion (a), we have
\begin{align*}
\delta_{{\mathcal P}^\bot} \leq \frac{m}{2}\left( \frac{m}{s}+1 \right) +\left( \frac{m}{s}+1\right) \cdot s
=\frac{m^2}{s} \left( \frac{1}{2}+\frac{3 s}{2 m}+\frac{s^2}{m^2}\right) \leq C^{\prime}\dfrac{m^{2}}{s},
\end{align*}
where the last inequality follows from the assumption by completing the square with respect to $s/m$.
\end{proof}

\begin{rmk}\normalfont
This remark is to clarify relations between the above result and existing results.
Fix $\alpha$, and consider the space of $\alpha$-smooth functions. 
For this (and even a larger) function class, Dick \cite[Corollary~5.5 and the comment after its proof]{3} gave
digital nets for which the QMC integration error is bounded from above by the order of $2^{-\alpha m}m^{\alpha s+1}$.
This is optimal, since for any point set of size $2^m$,
Sharygin \cite{11} constructed an $\alpha$-smooth function whose QMC integration error is at least of this order.

Since ${\rm WAFOM}$ gives only an upper bound of the QMC integration error,
our lower bound $2^{-C'm^2/s}$ on ${\rm WAFOM}$ in Theorem~\ref{low} implies nothing on the lower bound of the integration error. 

A merit of ${\rm WAFOM}$ is that the value depends only on the point set, not on the smoothness $\alpha$ such as \cite{3}.
On the other hand, ${\rm WAFOM}$ depends on the degree $n$ of discretization.
Thus, it seems not easy to compare directly the upper bound on the integration error given in \cite{3} and that by ${\rm WAFOM}$.
However, we might consider that our lower bound $2^{-C'm^2/s}$, which is independent of $n$ and $\alpha$,
shows a kind of limitation of the method in bounding the integration error in \cite{3} in the limit $\alpha \to \infty$.
\end{rmk}

\end{document}